\documentclass[12pt, a4paper]{amsart}

\usepackage[T1]{fontenc}
\usepackage[utf8]{inputenc}
\usepackage{graphicx}
\usepackage{float}
\usepackage[colorlinks=true, linkcolor=blue]{hyperref}
\usepackage{anyfontsize}
\usepackage{bbm}
\usepackage{amsmath}
\usepackage{amsthm}
\usepackage{amssymb}
\usepackage[shortlabels]{enumitem}
\usepackage[noabbrev]{cleveref}
\theoremstyle{plain}
\newtheorem{theorem}{Theorem}[section]
\newtheorem{fact}[theorem]{Fact}
\newtheorem{lemma}[theorem]{Lemma}
\newtheorem{corollary}[theorem]{Corollary}
\newtheorem{conjecture}{Conjecture}

\newtheorem{proposition}[theorem]{Proposition}

\theoremstyle{definition}

\newtheorem{definition}{Definition}

\newtheorem*{remark*}{Remark}
\newtheorem*{question*}{Question}

\DeclareMathOperator{\range}{ran}

\DeclareMathOperator{\baire}{\omega^\omega}

\DeclareMathOperator{\seqN}{\omega^{<\omega}}
\DeclareMathOperator{\seq2}{2^{<\omega}}
\DeclareMathOperator{\analytic}{\mathbf{\Sigma}_1^1}
\DeclareMathOperator{\coanalytic}{\mathbf{\Pi}_1^1}

\DeclareMathOperator{\infomega}{[\omega]^{\omega}}
\DeclareMathOperator{\IP}{IP}
\DeclareMathOperator{\concat}{^\frown}

\title{Ideal Analytic Sets}
\author{Łukasz Mazurkiewicz}
\email{lukasz.mazurkiewicz@pwr.edu.pl}

\author{Szymon Żeberski}
\email{szymon.zeberski@pwr.edu.pl}

\address{Łukasz Mazurkiewicz, Szymon Żeberski, Faculty of Pure and Applied Mathematics, Wrocław University of Science and Technology, 50-370 Wrocław, Poland}

\thanks{The work has been partially financed by grant {\bf 8211204601, MPK: 9130730000} from the Faculty of Pure and Applied Mathematics, Wrocław University of Science and Technology.
    \\
    Data sharing is not applicable to this article as no new data were created or analyzed in this study.
    \\
    {\bf Corresponding author}: Łukasz Mazurkiewicz, lukasz.mazurkiewicz@pwr.edu.pl
	\\
	AMS Classification 2020: Primary: 03E15,  28A05, 54H05; Secondary: 03E75.
	}

\begin{document}

    \maketitle
    \begin{abstract}
The aim of this paper is to give natural examples of $\analytic$-complete and $\coanalytic$-complete sets. 

In the first part, we consider ideals on $\omega$. We use a unified approach introduced in \cite{uogolnione_idealy} to create reductions of the collection of ill-founded trees to the ideals, proving $\analytic$-completeness of the ideals.

In the second part, we show the connection between this topic, families of trees and coding of $\sigma$-ideals of Polish spaces. In particular, we use the unified approach to prove that sets of codes for closed Ramsey-null sets, for closed $\sigma$-compact sets and for closed not strongly dominating sets are $\coanalytic$-complete.
\end{abstract}
    \section{Introduction}

We will use standard set theoretic notions (mostly following \cite{srivastava} and \cite{jech}). In particular, $\omega$ is the first infinite cardinal, i.e. the set of all natural numbers,  $\seqN$ and $\seq2$ are sets of finite sequences of elements of $\omega$ and $\{0,1\}$, respectively.
  For this part assume that $\mathcal{X}$ and $\mathcal{Y}$ are Polish spaces, i.e. separable completely metrizable topological spaces. Classical examples are the real line $\mathbb{R}$, the Baire space $\omega^\omega$, the Cantor space $2^\omega$.
\begin{definition}
    Let $A\subseteq \mathcal{X}$, $B\subseteq \mathcal{Y}$. We say that $B$ is \textit{Borel reducible} to $A$ if there is a Borel function $f: \mathcal{Y}\rightarrow \mathcal{X}$ satisfying $f^{-1}[A]=B$.
\end{definition}

\begin{definition}
    We say that $A\subseteq \mathcal{X}$ is \textit{$\analytic$-complete} if $A$ is analytic and for every Polish space $\mathcal{Y}$ every analytic set  $B\subseteq \mathcal{Y}$ is Borel reducible to $A$.
\end{definition}
\begin{fact}
    \label{analityczny_z_analitycznego}
    If an analytic set $B$ is Borel reducible to $A$ and $B$ is $\analytic$-complete, then $A$ is $\analytic$-complete.
\end{fact}

Note that the existence of analytic non-Borel sets and closure of Borel sets under Borel maps implies that all $\analytic$-complete sets are not Borel. Moreover, in order to apply \Cref{analityczny_z_analitycznego}, we need an example of a $\analytic$-complete set. Such an example can be found among trees over $\omega$.
\begin{definition}
    A set $T\subseteq\seqN$ is a \textit{tree over $\omega$} if
    $$\left(\forall \sigma\in\seqN\right)\left(\forall \tau\in\seqN\right)\left(\sigma\in T\land \tau\subseteq\sigma\implies\tau\in T\right).$$
    The set of all trees over $\omega$ will be denoted by $\textrm{Tree}_\omega$. The \textit{body} of a tree $T$ is
    $$[T]=\left\{x\in\baire: \left(\forall n\in\omega\right)\left(x\upharpoonright n\in T\right)\right\}.$$
\end{definition}

The set $\textrm{Tree}_\omega$ can be seen as a $G_\delta$ subset of $P(\seqN)$. Therefore, $\textrm{Tree}_\omega$ is a Polish space. By $\textrm{IF}_\omega$ let us denote the collection of all ill-founded trees over $\omega$, i.e. all trees with non-empty body. It turns out that $\textrm{IF}_\omega$ is an example of a $\analytic$-complete set(see, e.g., \cite[Example 4.2.1]{srivastava}).

\begin{definition}
 We say that $A\subseteq \mathcal{X}$ is \textit{$\coanalytic$-complete} if $\mathcal{X}\setminus A$ is $\analytic$-complete.
\end{definition}

In the first part of this paper we will consider ideals on natural numbers. 

\begin{definition}
    A family $\mathcal{I}\subseteq P(\omega)$ is an \textit{ideal} if $\mathcal{I}$ is closed
    \begin{itemize}
        \item under taking subsets, i.e $(\forall I\in\mathcal{I})(\forall J\subseteq I)(J\in\mathcal{I})$,
        \item under finite unions, i.e. $(\forall I,J\in\mathcal{I})(I\cup J\in\mathcal{I})$.
    \end{itemize}
    We say that $\mathcal{I}$ is proper if $\omega\not\in\mathcal{I}$.
\end{definition}
The space $P(\omega)$ is naturally (via characteristic functions $A\mapsto\chi_A$) homeomorphic to the Cantor space $2^\omega$. Therefore, every ideal $\mathcal{I}$ on $\omega$ can be seen as a subset of the Cantor space. In this context, questions about Borel complexity, being $\analytic$-complete or being $\coanalytic$-complete, make sense for $\mathcal{I}$.
    \section{Ideals on $\omega$}

In this section we will follow the simplified version of the notation introduced in \cite{uogolnione_idealy}. Let $\Lambda$ be a countable infinite set and $\rho: \mathcal{P}(\omega)\rightarrow\mathcal{P}(\Lambda)$ be a monotone function. We define a family
$$\mathcal{I}_\rho=\left\{A\subseteq\Lambda: (\forall F\in[\omega]^\omega) (\rho(F)\nsubseteq A\right)\}.$$

In \cite{uogolnione_idealy}, authors give necessary conditions for $\mathcal{I}_\rho$ to be an ideal and to be a coanalytic set. Moreover, they show that (using more general approach, which we omit for simplicity) every ideal can be represented in a similar way.

\subsection{Unified approach to reduction}\hfill\\

In the following subsections we will give examples of families, defined as above, which are $\coanalytic$-complete. In order to do that we construct reductions of $\textrm{IF}_\omega$. Since they are very similar to each other, it makes sense to define the reduction here and only later prove that they work for specific examples.

Fix a function $\rho$ as earlier. We would like to construct a Borel function $\varphi: \textrm{Tree}_\omega\rightarrow\mathcal{P}(\Lambda)$ such that $\varphi^{-1}[\mathcal{I}_\rho^+]=\textrm{IF}_\omega$.
Let us first fix an injection $\alpha: \seqN\rightarrow\omega$ satisfying
\begin{equation}\label{warunek_alfa}
    \sigma\subseteq\tau\Rightarrow\alpha(\sigma)\leq\alpha(\tau).
\end{equation}
Now define
$$\varphi(T)=\bigcup_{\sigma\in T}\rho(\{\alpha(\sigma\upharpoonright k):\;\; k\leq|\sigma|\}).$$
From monotonicity of $\rho$, if $\tau\in[T]$, then $\rho(\{\alpha(\tau\upharpoonright k): k\in\omega\})\subseteq\varphi(T)$, so image of an ill-founded tree is in $\mathcal{I}_\rho^+$. Therefore, the only thing left is to prove that if $\rho(B)\subseteq\varphi(T)$ for some $B\in[\omega]^\omega$, then $[T]\ne\emptyset$.

\subsection{Ramsey ideal}
\begin{definition}
    Define $r: [\omega]^\omega\rightarrow[[\omega]^2]^\omega$ by
    $$r(B)=[B]^2.$$
    $\mathcal{R}=\mathcal{I}_r$ is called the \textit{Ramsey ideal}.
\end{definition}

In \cite[Theorem 3.3]{barnabas} the following result was already proven. However, we would like to give a proof based on the unified approach.
\begin{theorem}
    $\mathcal{R}$ is $\coanalytic$-complete.
\end{theorem}
\begin{proof}
    Suppose that $[B]^2\subseteq\varphi(T)$ for some $B\in[\omega]^\omega$. Take any $a,b\in B$, $a<b$. Then $\{a,b\}\in\varphi(T)$, hence (from definition of $\varphi$), there is $\sigma\in T$ such that $\alpha^{-1}(a)=\sigma\upharpoonright k$ and $\alpha^{-1}(b)=\sigma\upharpoonright l$ for some $k,l\leq|\sigma|$, $k\ne l$. Thus, because of condition (\ref{warunek_alfa}), $\alpha^{-1}(a)\subset\alpha^{-1}(b)$ and $\alpha^{-1}(a),\alpha^{-1}(b)\in T$. Using this argument we can find a sequence $(a_n)_{n\in\omega}$ of elements of $B$ such that
    $$i<j\implies \alpha^{-1}(a_i)\subset\alpha^{-1}(a_j)$$
    and $\alpha^{-1}(a_i)\in T$ for every $i\in\omega$. That generates an infinite branch in tree $T$.
\end{proof}

\subsection{Hindman ideal}
\begin{definition}
Define $\textrm{FS}:[\omega]^\omega\rightarrow[\omega]^\omega$ by
$$\textrm{FS}(B)=\left\{\sum_{n\in F}n:\;\; F\subseteq B\text{ is nonempty and finite}\right\}.$$
A set $A\subseteq\omega$ is called an IP-set if there is an infinite $B\subseteq A$ satisfying $\textrm{FS}(B)\subseteq A$.

The \textit{Hindman ideal} is the family $\mathcal{H}$ of all non-IP-sets, i.e. $\mathcal{H}=\mathcal{I}_{\textrm{FS}}$. As proven in \cite[Proposition 5.2.(5)]{uogolnione_idealy}, the Hindman ideal is coanalytic. The theorem below is a strengthening of that result.
\end{definition}

\begin{theorem} \label{hindman_complete}
$\mathcal{H}$ is $\coanalytic$-complete.
\end{theorem}
\begin{proof}
First, we need an additional condition on the function $\alpha$, namely that $\alpha: \seqN\rightarrow\left\{2^{2n}: n\in\omega\right\}$.

Suppose that $\varphi(T)$ is an IP-set, i.e. $\textrm{FS}(B)\subseteq\varphi(T)$ for some $B\in\infomega$. Note that $B\subseteq \textrm{FS}(B)$ and, because of the definition of $\varphi$, $\varphi(T)$ contains only numbers which, in binary expansion, have "$1$" at even positions. From now on we will identify natural numbers with their binary expansions. Take any $a_0\in B$ and consider the set $B_0=B\setminus\{b\in B: b>a_0\}$. Take any $b_0\in B_0$. From the above remark, $a_0$ and $b_0$ have "$1$"s at even positions. Moreover, the positions of "$1$"s in $a_0$ are disjoint from the positions of "$1$"s in $b_0$ (otherwise $a_0+b_0\in \textrm{FS}(B)$ has $1$ at an odd position). Put $a_1=a_0+b_0$, consider the set $B_1=B\setminus\{b\in B: b>a_1\}$ and take any $b_1\in B_1$. With the same argument as before, $b_1$ and $a_1$ have "$1$"s at disjoint positions. So we put $a_2=a_1+b_1$ and go on with the construction.

As a result we obtain a sequence $(a_n)_{n\in\omega}$ of natural numbers satisfying
\begin{equation*}
    a_n\text{ has 1 at position $k$}\Rightarrow a_{n+1}\text{ has 1 at position $k$}
\end{equation*}
for every $n\in\omega$. This sequence induces a sequence $(2^{c_n})_{n\in\omega}$ of increasing powers of $2$ (by decomposing elements of $(a_n)_{n\in\omega}$ into sums of powers of $2$) contained in $\varphi(T)$. Now note that, by definition of $\varphi$ and condition (\ref{warunek_alfa}),
\begin{equation*}
    \alpha^{-1}(2^{c_n})\in T\;\land\;\alpha^{-1}(2^{c_n})\subseteq\alpha^{-1}(2^{c_{n+1}})
\end{equation*}
holds for every $n\in\omega$. Hence, $(\alpha^{-1}(2^{c_n}))_{n\in\omega}$ is an increasing sequence of elements of $T$, so $[T]\ne\emptyset$.
\end{proof}

\subsection{Modifications of the Hindman ideal}
\begin{definition}
For $k\ge 2$ define the function $\textrm{FS}_k:[\omega]^\omega\rightarrow[\omega]^\omega$ by
$$\textrm{FS}_k(B)=\left\{\sum_{n\in F}n:\;\; |F|=k\right\}.$$
A set $A\subseteq\omega$ is called an $\textrm{IP}_k$-set if there is an infinite $B\subseteq A$ satisfying $\textrm{FS}_k(B)\subseteq A$.
\end{definition}

From now on let us fix a natural number $k\ge 2$.

Analogously as before, $\mathcal{H}_k$ will denote a family of all non-$\IP_k$-sets, i.e. $\mathcal{H}_k=\mathcal{I}_{\textrm{FS}_k}$ 

\begin{fact}
    $\mathcal{H}_k$ forms an ideal.
\end{fact}
\begin{proof}
    Closure under taking subsets is clear. Let $A,B\subseteq\omega$ be non-$\IP_k$-sets (without loss of generality disjoint). Suppose $A\cup B$ is an $\IP_k$-set, i.e. there is $C\in\infomega$ such that $\textrm{FS}_k(C)\subseteq A\cup B$. Consider a function $c: [C]^k\rightarrow\{A,B\}$ defined with formula
    $$c(E)=\begin{cases}
            A,&\mbox{if } \sum_{n\in E}n\in A\\
            B,&\mbox{if } \sum_{n\in E}n\in B\\
            \end{cases}$$
    Using Ramsey's Theorem, there is an infinite $D\subseteq C$ such that all $k$-element subsets of $D$ are "colored" by $c$ with the same color, without loss of generality $A$. Thus, $\textrm{FS}_k(D)\subseteq A$, which gives us a contradiction.
\end{proof}

Obviously, every IP-set is $\IP_k$ set, so $\mathcal{H}_k$ is a subideal of $\mathcal{H}$. This way we obtain an infinite collection of ideals over $\omega$. We shall first focus on the structure of this collection.

\begin{fact}
    Let $n,m\in\omega$. Then
    $$n\mid m\iff \mathcal{H}_n\subseteq\mathcal{H}_m.$$
\end{fact}
\begin{proof}
    Assume $n\mid m$. Take $k\in\omega$ such that $m=nk$. Suppose $A$ is an $\IP_m$-set, i.e. there is an infinite set $B\subseteq\omega$ satisfying $\textrm{FS}_m(B)\subseteq A$. Since $B$ is infinite, there is $p\in\omega$ such that the set $$B_p=\{a\in B: a\equiv p\mod n\}$$ is infinite. Fix $a_0,a_1,\ldots,a_{n(k-1)-1}\in B_p$. Note that $n\mid\sum_{i<n(k-1)}a_i$, so
    $b=\frac{1}{n}\sum_{i<n(k-1)}a_i$ is a natural number.
    Now, consider the sets $B'=B\backslash\{a_0,a_1,\ldots,a_{n(k-1)-1}\}$ and $B''=B'+b=\{c+b: c\in B'\}$. Take any $b_0,b_1,\ldots,b_{n-1}\in B''$. There are $c_0,c_1,\ldots,c_{n-1}\in B'$ such that $b_i=c_i+b$ for every $i<n$. Then
    $$\sum_{i<n}b_i=\sum_{i<n}(c_i+b)=bn+\sum_{i<n}c_i=\sum_{j<n(k-1)}a_j+\sum_{i<n}c_i\in FS_m(B)$$
    as a sum of $n(k-1)+n=nk=m$ elements from the set $B$. Therefore, $\textrm{FS}_n(B'')\subseteq \textrm{FS}_m(B)\subseteq A$, so $A$ is an $\IP_n$-set.

    Assume now that $n\nmid m$. Consider $A=\{k\in\omega: n\nmid k\}$. Observe that
    $$\textrm{FS}_m(\{kn+1: k\in\omega\})\subseteq A,$$
    so $A$ is an $\IP_m$-set. However, it is not an $\IP_n$-set. Suppose it is. There is an infinite $B\subseteq\omega$ such that $\textrm{FS}_n(B)\subseteq A$. Since $B$ is infinite, we can (like before) find $a_0,a_1,\ldots,a_{n-1}$ such that $a_i\equiv a_j\mod n$ for each $i,j<n$. But then $\sum_{i<n}a_i\in FS_n(B)\subseteq A$ is divisible by $n$, which yields a contradiction.
\end{proof}
\begin{corollary}
    For each $k\in\omega$, $\mathcal{H}_k$ is a proper subideal of $\mathcal{H}$.
\end{corollary}

It is also rather easy to see that $\textrm{FS}_k$ is continuous for every $k\in\omega$ (similarly to continuity of $\textrm{FS}$, which can be found in \cite[Proposition 5.2.(2)]{uogolnione_idealy}). Hence, from \cite[Proposition 5.1]{uogolnione_idealy}, $\mathcal{H}_k$ is coanalytic for every $k\in\omega$.

\begin{theorem}
    $\mathcal{H}_2$ is $\coanalytic$-complete.
\end{theorem}
\begin{proof}
    Take $\alpha$ like in the proof of \Cref{hindman_complete}.

    Suppose now that $\varphi(T)$ is an $\IP_2$-set, i.e. $\textrm{FS}_2(B)\subseteq\varphi(T)$ for some $B\in\infomega$. From now on, we identify natural numbers with their binary expansions and $|n|$ denotes the length of $n$ in this expansion. For any natural number $n$, $n=\sum_{i<|n|}n_i2^i$. Note that $\varphi(T)$ contains only numbers which have exactly two "$1$"s. Moreover, these "$1$"s have to be at even positions.

    We claim that $B\cap \{2^k: k\in\omega\}$ is infinite. Suppose it is not, so without loss of generality $B$ does not contain any power of $2$. Consider the function $f_1(n)=\min\{i: n_i=1\}$. If $f_1[B]$ is infinite, take any $a\in B$. Choose $b\in B$ such that $f_1(b)>2|a|$. But then $a+b\in\textrm{FS}_2(B)$ and $a+b$ has more than two "$1$"s in binary expansion. Hence, $f_1[B]$ is finite, so without loss of generality $f_1[B]=\{k_1\}$. Furthermore, all elements of $B$ have the same number at position $k_1+1$. If this number is $1$, we get a contradiction since
    $$(\ldots01\underset{k_1}{1}\ldots)_2+(\ldots01\underset{k_1}{1}\ldots)_2=(\ldots11\underset{k_1}{0}\ldots)_2,$$
    $$(\ldots11\underset{k_1}{1}\ldots)_2+(\ldots11\underset{k_1}{1}\ldots)_2=(\ldots11\underset{k_1}{0}\ldots)_2,$$
    so the sum of elements of $B$ has "$1$"s in adjacent positions (we can once again assume without loss of generality that all elements of $B$ have the same number at position $k_1+2$ and are bigger than $2^{k_1+3}$). Thus, all elements of $B$ have $0$ at position $k_1+1$. By repeating the above argument for $f_2=\min\{i>k_1: n_i=1\}$, all elements of $B$ have $1$ at some position $k_2>k_1$ and $0$ at position $k_2+1$ and are bigger than $2^{k_2+3}$. But
    $$(1\ldots0\underset{k_2}{1}\ldots0\underset{k_1}{1}\ldots)_2+(1\ldots0\underset{k_2}{1}\ldots0\underset{k_1}{1}\ldots)_2=(1\ldots1\underset{k_2}{0}\ldots1\underset{k_1}{0}\ldots)_2,$$
    which again leads to a contradiction as elements of $B$ sum up to a number having more than two "$1$"s.

    Define $B'=B\cap \{2^{2k}: k\in\omega\}$. Then $\textrm{FS}_2(B')\subseteq\textrm{FS}_2(B)\subseteq\varphi(T)$. Take any $a,b\in B'$, $a<b$. $a+b\in\varphi(T)$, so (according to the definition of $\varphi$) there are $\sigma\in T$ and $k,l<|\sigma|$ such that $\alpha^{-1}(a)=\sigma\upharpoonright k$ and $\alpha^{-1}(b)=\sigma\upharpoonright l$. Therefore (from condition (\ref{warunek_alfa})) $\alpha^{-1}(a)\subset\alpha^{-1}(b)$ and $\alpha^{-1}(a),\alpha^{-1}(b)\in T$. Using this argument we can find a sequence $(a_n)_{n\in\omega}$ of elements of $B'$ such that
    $$i<j\implies \alpha^{-1}(a_i)\subset\alpha^{-1}(a_j)$$
    and $\alpha^{-1}(a_i)\in T$ for every $i\in\omega$. That generates an infinite branch through $T$.
\end{proof}

\begin{conjecture}
    $\mathcal{H}_k$ is $\coanalytic$-complete for every $k\ge 2$.
\end{conjecture}

Let us give some intuition why the proof presented above does not directly translate to all $k$ \footnote{While the paper was reviewed, Jacek Tryba e-mailed us a proof of the conjecture for any $k\geq 2$.}. There are two possible ways: we either stick to the binary expansion of numbers (maybe increasing gaps between "$1$"s) or we consider the $k$-adic expansion. Either way, it is not clear how to achieve a contradiction similar to
$$(\ldots01\underset{k_1}{1}\ldots)_2+(\ldots01\underset{k_1}{1}\ldots)_2=(\ldots11\underset{k_1}{0}\ldots)_2,$$
    $$(\ldots11\underset{k_1}{1}\ldots)_2+(\ldots11\underset{k_1}{1}\ldots)_2=(\ldots11\underset{k_1}{0}\ldots)_2.$$
When we keep the binary expansion, there is no control over how close the resulting "$1$"s are to each other (it is quite easy to see in case $k=3$). When we change to $k$-adic expansion, there are more possibilities to consider on every position.

\subsection{Differences instead of sums}
\begin{definition}
Define $\Delta: [\omega]^\omega\rightarrow[\omega]^\omega$ by
$$\Delta(B)=\left\{a-b: a,b\in B\land a>b\right\}.$$
A set $A\subseteq\omega$ is called a \textit{D-set} if there is an infinite $B\subseteq\omega$ satisfying $\Delta(B)\subseteq A$.
\end{definition}

Let $\mathcal{D}$ denote the family of all non-D-sets. As proved in \cite{hindmann_d_idealy}, $\mathcal{D}$ is a proper subideal of $\mathcal{H}$. Moreover, it is coanalytic (again due to \cite[Proposition 5.2.(5)]{uogolnione_idealy}).

\begin{theorem}
    $\mathcal{D}$ is $\coanalytic$-complete.
\end{theorem}
\begin{proof}
    This time we assume that $\alpha: \seqN\rightarrow\{2^n: n\in\omega\}$.

    Suppose now that $\Delta(B)\subseteq\varphi(T)$ for some infinite $B\subseteq\omega$. Take $n_0=\min\{k: (\exists n\in\omega)(2^n-2^k\in \Delta(B))\}$. Without loss of generality we can assume that $2^{n_0}\in B$ (otherwise if we consider $b=\min B$ and $B'=\{a+(2^{n_0}-b): a\in B\}$, $\Delta(B)=\Delta(B')$). Observe that now all elements of the set $B$ are powers of $2$ (because only differences of powers of $2$ are in $\varphi(T)$). Take any $2^{n_1}\in B$. $2^{n_1}-2^{n_0}\in \Delta(B)$, so (from the definition of $\varphi$ and condition (\ref{warunek_alfa})) $\alpha^{-1}(2^{n_0})\subseteq\alpha^{-1}(2^{n_1})$ and $\alpha^{-1}(2^{n_1})\in T$. Now take any $2^{n_2}\in B$. Analogously as above, $2^{n_2}-2^{n_1}\in \Delta(B)$, hence $\alpha^{-1}(2^{n_1})\subseteq\alpha^{-1}(2^{n_2})$ and $\alpha^{-1}(2^{n_2})\in T$.
    
    Repeating this argument we obtain a sequence $(n_k)_{k\in\omega}$ such that
    $$(\forall k\in\omega)(\alpha^{-1}(2^{n_k})\subseteq\alpha^{-1}(2^{n_{k+1}})\land\alpha^{-1}(2^{n_k})\in T).$$
    Therefore, $(\alpha^{-1}(2^{n_k}))_{k\in\omega}$ is an increasing sequence of elements of $T$, so $[T]\ne\emptyset$.
\end{proof}

    \section{Trees and ideals}

As noted in the introduction, a tree over a countable set $A$ can be seen as a point in the space $\mathcal{P}(A^{<\omega})$, which is homeomorphic to $\mathcal{P}(\omega)$. Therefore, families of trees, as subsets of this space, have similar topological structure to ideals on $\omega$. In the second part of this work we will investigate the descriptive complexity of classical families of trees, like Miller or Laver trees. Let us start with a theorem connecting this topic with earlier results.
\begin{theorem}\label{mathias_trees}
    Consider the function $M: \mathcal{P}(\omega)\rightarrow\mathcal{P}(\seqN)$ given by
    $$M(B)=\{\sigma\in B^{<\omega}: \sigma\text{ is increasing}\}.$$
    $\mathcal{I}_M$ is $\coanalytic$-complete.
\end{theorem}
\begin{proof}
    First, notice that $\range\varphi\subseteq\textrm{Tree}_\omega$. Therefore, since $\varphi[\textrm{IF}_\omega]\subseteq\mathcal{I}_M^+$, it is sufficient to show that $\varphi^{-1}[\textrm{IF}_\omega]=\textrm{IF}_\omega$.

    Suppose that $\tau\in[\varphi(T)]$. Take any $k\in\omega$. By definition, $\tau\upharpoonright (k+2)\in\varphi(T)$. Hence, there is $\sigma\in T$ such that $\alpha(\sigma\upharpoonright k_1)=\tau(k)$ and $\alpha(\sigma\upharpoonright k_2)=\tau(k+1)$ for some $k_1,k_2\in\omega$. But $\tau(k+1)>\tau(k)$ (since $\tau\upharpoonright (k+2)\in\varphi(T)$ is increasing), so because of condition (\ref{warunek_alfa}) $k_1<k_2$. Therefore, $\alpha^{-1}(\tau(k))\subset\alpha^{-1}(\tau(k+1))$, so $(\alpha^{-1}(\tau(k)))_{k\in\omega}$ is an increasing sequence of elements of $T$, so $[T]\ne\emptyset$.
\end{proof}

Clearly, $\mathcal{I}_M$ is not a family of trees ($A\in\mathcal{I}_M$ does not necessarily have to be closed under taking subsequences). However, since $\range\varphi\subseteq\textrm{Tree}_\omega$, the proof of \Cref{mathias_trees} can be easily repeated for trees containing $M(B)$ for some $B\in[\omega]^\omega$ instead of $\mathcal{I}_M^+$ (their preimages via $\varphi$ are the same).
\begin{corollary}
    The family of all trees containing $M(B)$ for some $B\in[\omega]^\omega$ is $\analytic$-complete.
\end{corollary}
\begin{definition}
    We call a tree of the form $M(B)$ for some $B\in[\omega]^\omega$ a \textit{Mathias bush}.

    Moreover, if $T_\sigma=\{\tau\in\seqN: \max(\sigma)<\min(\tau)\land\sigma\concat\tau\in T\}$ is a Mathias bush for some $\sigma\in T$, we call $T$ a \textit{Mathias tree}.
\end{definition}
The idea of the above definitions comes from Mathias forcing. Mathias forcing consists of pairs $(s, S)$, where $s\in\seqN$ is increasing and $S\in[\omega]^\omega$ such that $\max(\textrm{rng}(s))<\min(S)$. The ordering of pairs is defined as follows:
$$(s,S)\geq(t,T)\iff s\subseteq t \land S\supseteq T\land \textrm{rng}(t)\backslash\textrm{rng}(s)\subseteq S\backslash T.$$
Hence, if we fix a pair $(s,S)$, we can think of a Mathias tree as a tree with stem $s$ and all possible $t\in\seqN$ such that $(s,S)\geq(t,T)$ for some $T\in[\omega]^\omega$.

Now note that in the proof of \Cref{mathias_trees}, having a stem or not does not make any difference. Therefore, we can conclude the following result.
\begin{corollary}
    The family of all trees containing a Mathias tree is $\analytic$-complete.
\end{corollary}

Recall that every tree over $\omega$ codes a closed subset of $\baire$ (via the body of the tree). Moreover, a closed set contains a body of a Mathias tree if and only if it is not Ramsey-null (\cite[Section 4.7.7.]{zapletal}). Thus, we can formulate the above result in terms of codes of closed subsets of the Baire space belonging to the $\sigma$-ideal of Ramsey-null sets.

\begin{corollary}
    The set of codes for closed Ramsey-null sets is $\coanalytic$-complete.
\end{corollary}

That result corresponds to the one obtained in \cite{sabok}, where it is proved that the set of codes for analytic Ramsey-null sets is $\mathbf{\Pi}_2^1$-complete.

\begin{definition}
    A tree $T\in\textrm{Tree}_\omega$ is called \textit{Miller} if
    $$(\forall \sigma\in T)(\exists\tau\supseteq\sigma)(\exists^\infty n\in\omega)(\tau\concat n\in T).$$
    A tree $T\in\textrm{Tree}_\omega$ is called \textit{Laver} if
    $$(\exists \sigma\in T)(\forall\tau\supseteq\sigma)(\exists^\infty n\in\omega)(\tau\concat n\in T).$$
\end{definition}

Clearly, all Mathias trees are both Miller and Laver. Hence, repeating the argumentation used for Mathias trees, in the proof of \Cref{mathias_trees} we can put the family of all trees containing a Miller (or Laver) tree in place of $\mathcal{I}_M^+$ (since $\varphi[\textrm{IF}_\omega]\subseteq\mathcal{I}_M^+$ and $\varphi^{-1}[\textrm{IF}_\omega]=\textrm{IF}_\omega$, the preimage of any Miller tree which is not Mathias via $\varphi$ is empty). This will give us
\begin{corollary}
    \begin{enumerate}
        \item The family of all trees containing a Miller tree is $\analytic$-complete.
        \item The family of all trees containing a Laver tree is $\analytic$-complete.
    \end{enumerate}
\end{corollary}

Similarly to the case of Mathias trees, we can also present the above results in terms of codes for some $\sigma$-ideals of subsets of the Baire space. A subset of the Baire space contains a body of a Miller tree if and only if it is not $\sigma$-compact (\cite{kechris}) and contains a body of a Laver tree if and only if it is strongly dominating (\cite[Lemma 2.3]{TreeIdeals}).

\begin{corollary}
    \begin{enumerate}
        \item The set of codes for closed $\sigma$-compact sets is $\coanalytic$-complete.
        \item The set of codes for closed not strongly dominating sets is $\coanalytic$-complete.
    \end{enumerate}
\end{corollary}

We would also like to explore some known families of trees over the set $\{0,1\}$. Like before, such trees can be seen as points in the space $\mathcal{P}(\seq2)$, homeomorphic to $\mathcal{P}(\omega)$ (from which we can derive a Polish space $\textrm{Tree}_2$). However, unlike in case of $\omega$, $\textrm{IF}_2$ is a $G_\delta$ subset of $\textrm{Tree}_2$.

\begin{definition}
    A tree $T\in\textrm{Tree}_2$ is called a \textit{Sacks tree} (or \textit{perfect tree}) if
    $$(\forall\sigma\in T)(\exists\tau\in T)(\tau\supseteq\sigma\land\tau\concat0\in T\land\tau\concat1\in T).$$
    
    A tree $T\in\textrm{Tree}_2$ is called a \textit{Silver tree} if there exist a coinfinite set $A\subseteq\omega$ and a function $x:\omega\rightarrow 2$ such that
    $$T=\left\{\sigma\in\seq2: (\forall n\in|\sigma|\cap A)(\sigma(n)=x(n)\right\}.$$
    Then we call $A$ the \textit{domain} of $T$ (and $\omega\backslash A$ the \textit{co-domain}) and $x$ the \textit{Silver function} of $T$.
\end{definition}

\begin{theorem}
    The family $\mathcal{S}$ of all trees containing a Sacks tree is $\analytic$-complete.
\end{theorem}

    For the proof we will need the following lemma, the proof of which can be found in \cite[Proposition 4.2.5]{srivastava}.
    \begin{lemma}
        Let $X$ be an uncountable Polish space. Then
        $$U(X)=\left\{K\in\ K(X): |K|>\omega\right\}$$
        is $\analytic$-complete, where $K(X)$ is the space of all compact subsets of $X$ equipped with the Vietoris topology.
    \end{lemma}
\begin{proof}
    (of the theorem) Define $\varphi: K(2^\omega)\rightarrow\textrm{Tree}_2$ by
    $$\varphi(K)=\left\{\sigma\upharpoonright k: \sigma\in K, k\in\omega\right\}.$$
    We need to show that $\varphi$ is Borel and $\varphi^{-1}[\mathcal{S}]=U(2^\omega)$. Let us start with the second assertion. By $\mathbb{S}a$ denote the family of all Sacks trees.
    \begin{align*}
        \varphi^{-1}[\mathcal{S}]&=\left\{K\in K(2^\omega): \varphi(K)\in\mathcal{S}\right\}\\
        &=\left\{K\in K(2^\omega): (\exists T\in\mathbb{S}a)([T]\subseteq K)\right\}\\
        &=\left\{K\in K(2^\omega): |K|>\omega\right\}=U(2^\omega),
    \end{align*}
    where the second to last equality comes from the perfect set property for closed subsets of $2^\omega$.

    Now, fix $s\in\seq2$. With $[s]_{\textrm{Tree}_2}$ and $[s]_{2^\omega}$ denote the subbase clopen sets generated by $s$ in the spaces $\textrm{Tree}_2$ and $2^\omega$, respectively, i.e. $$[s]_{\textrm{Tree}_2}=\left\{T\in\textrm{Tree}_2: s\in T\right\}, [s]_{2^\omega}=\left\{\sigma\in 2^\omega: \sigma\supseteq s\right\}.$$
    \begin{align*}
        \varphi^{-1}[[s]_{\textrm{Tree}_2}]&=\left\{K\in K(2^\omega): (\exists\sigma\in K)(\exists k\in\omega)(s=\sigma\upharpoonright k)\right\}\\
        &=\left\{K\in K(2^\omega): K\cap[s]_{2^\omega}\ne\emptyset\right\},
    \end{align*}
    which is an open set in the Vietoris topology. Therefore, the function $\varphi$ is continuous.
\end{proof}
\begin{theorem}
    The family $\mathcal{V}$ of all trees containing a Silver tree is $\analytic$-complete.
\end{theorem}
\begin{proof}
    We construct a Borel function $\varphi: \textrm{Tree}_\omega\rightarrow\textrm{Tree}_2$ such that $\varphi^{-1}[\mathcal{V}]=\textrm{IF}_\omega$. First, define some auxiliary functions. Let $\delta: \omega\rightarrow\seq2$ be the binary expansion of natural numbers, $\beta:\seq2\rightarrow\seq2$ be defined by
    $$\beta((a_0a_1\ldots a_k))=(a_0a_0a_1a_1\ldots a_ka_k)$$
    and $\gamma:\seqN\rightarrow\mathcal{P}(\seq2)$ be defined by
    \begin{multline*}
        \gamma((n_0n_1\ldots n_k))=\{\beta(\delta(n_0))\concat01i_001\concat\ldots\concat\beta(\delta(n_k))\concat01i_k01: \\i_0,i_1\ldots,i_k\in\{0,1\}\}.
    \end{multline*}
    Now we can define the function $\varphi':\textrm{Tree}_\omega\rightarrow\mathcal{P}(\seq2)$:
    $$\varphi'(T)=\bigcup_{\sigma\in T}\gamma(\sigma).$$
    As we can see, $\varphi'(T)$ is not a tree, because it is not closed under taking subsequences. However, we can fix it by "treefying" a result, i.e., define
    $$\psi(A)=\left\{\sigma\upharpoonright k: \sigma\in A, k\leq|\sigma|\right\}$$
    $$\varphi=\psi\circ\varphi'.$$
    Clearly, if $\tau=(\tau_0\tau_1\tau_2\ldots)\in[T]$ then $\varphi(T)$ contains a Silver tree with a Silver function $$x=\beta(\delta(\tau_0))\concat01\underline{0}01\concat\beta(\delta(\tau_1))\concat01\underline{0}01\concat\ldots$$ and co-domain the set of indices of underlined "$0$"s.

    Suppose now that $\varphi(T)$ contains a Silver tree $S$. The Silver function of $S$ has to be of the form $$x_S=\beta(\delta(\tau_0))\concat01i_001\concat\beta(\delta(\tau_1))\concat01i_101\concat\ldots$$ for some $\tau=(\tau_0\tau_1\tau_2\ldots)$ and $i_0,i_1,\ldots\in\{0,1\}$ (because no other sequences are produced by $\varphi$). Note that if the co-domain of $S$ is a subset of positions of "$i_k$"s in $x_S$, $\tau\in[T]$ (and $\tau$ can be uniquely determined from $x_S$). Assume then that there is another $m$ in the co-domain of $S$ (denote the co-domain of $S$ by $A_S$). Without loss of generality $m\leq |\beta(\delta(\tau_0))\concat01i_001|$ (the argument can be repeated inside any block of the form $\beta(\delta(\tau_k))\concat01i_k01$).

    First, note that $m\neq|\beta(\delta(\tau_0))|+4$ and $m\neq|\beta(\delta(\tau_0))|+5$ (i.e. it can not be any of the last two positions), since $\varphi$ fixes values at two positions after the block of the form $\beta(n)\concat01i$. Consider the case $m=|\beta(\delta(\tau_0))|+2$. As $m\in A_S$, $\beta(\delta(\tau_0))\concat01, \beta(\delta(\tau_0))\concat00\in S$. Since in a Silver tree every sequence can be extended to one which follows the Silver function, $\beta(\delta(\tau_0))\concat00i_001\concat\beta(\delta(\tau_1))\concat0\in S$. But then in $S$ there is a sequence of the form $\beta(\delta(\tau_0))\concat(00)^k\concat(11)^l\concat10$ (for some $k,l\in\omega$), which is not possible by the form of the function $\varphi$.

    Now consider the case $m=|\beta(\delta(\tau_0))|+1$. Because $m\in A_S$, $\beta(\delta(\tau_0))\concat0,\beta(\delta(\tau_0))\concat1\in S$. But then, similarly as in the previous case, $\beta(\delta(\tau_0))\concat11i_001\beta(\delta(\tau_1))0\in S$. Therefore, a sequence of the form $\beta(\delta(\tau_0))\concat11\concat(00)^k\concat(11)^l\concat10$ (for some $k,l\in\omega$) is in $S$, which is again impossible.

    There are two more possibilities for $m\leq|\beta(\delta(\tau_0))|$: either $2\mid m$ or $2\nmid m$. Let us focus on the second one (the first one can be analyzed in the similar way). If $\beta(\delta(\tau_0))$ has $1$ at the position $m$, it has $1$ also at the position $m-1$. But then the sequence of the form $\beta(n)\concat10$ (for some $n\in\omega)$ is in $S$, which leads to a contradiction. If $\beta(\delta(\tau_0))$ has $0$ at positions $m$ and $m-1$, $\beta(n)\concat00,\beta(n)\concat01\in S$ for some $n\in\omega$ and $\beta(n)\concat00\subseteq\beta(\delta(\tau_0))$. There is $\upsilon\in\seq2$, consisting only of blocks $00$ and $11$, such that $\beta(\delta(\tau_0))=\beta(n)\concat00\concat\upsilon$. Then $\beta(n)\concat01\concat\upsilon\concat01\in S$, so either $\beta(n)\concat0101\in S$, $\beta(n)\concat0111\in S$, $\beta(n)\concat010001\in S$ (which all lead to immediate contradiction) or $\beta(n)\concat010011\concat\mu\concat01\in S$, where $\upsilon=0011\concat\mu$. Hence, there is $a\in\omega$ such that either $1\concat\mu\concat0\subseteq\beta(a)$ or $\beta(a)\subseteq1\concat\mu\concat0$. The first case leads to the situation that $\beta(a)$ contains the block "$10$" with $1$ at an even position, which is not possible. However, in the second one $\beta(a)\concat10\subseteq1\concat\mu\concat0$, which is again impossible, because then $\beta(n)\concat01001\concat\beta(a)\concat10\in S$.

\end{proof}

Once again, containing a body of a Silver tree can be expressed in terms of being positive in regard to a $\sigma$-ideal of subsets of the Cantor space, namely the $\sigma$-ideal $I_G$ generated by Borel $G$-independent sets (where $G$ is a graph on $2^\omega$, \cite[Fact 4.7.16]{zapletal}).

\begin{corollary}
    The set of codes for closed $I_G$ sets is $\coanalytic$-complete.
\end{corollary}

Let us finish by showing that there are classical ideals, closed sets of which are not coded by $\coanalytic$-complete sets.

\begin{proposition}
    The set of codes for closed meager subsets of the Cantor (or Baire) space is Borel.
\end{proposition}
\begin{proof}
    Since the below argument can be directly rewritten for the Baire space, we will focus on the Cantor space. Every closed meager subset of $2^\omega$ can be coded by a tree $T$ such that
    $$(\forall\sigma\in\seq2)(\exists\tau\in\seq2)(\tau\supseteq\sigma\land\tau\notin T).$$
    Clearly, the set of all trees $T$ satisfying the above condition is a $G_\delta
    $ subset of $\mathcal{P}(\seqN)$.
\end{proof}
\begin{proposition}
    The set of codes for closed measure zero subsets of the Cantor space is Borel.
\end{proposition}
\begin{proof}
    A set $A\subseteq\seq2$ is null if for every $\varepsilon>0$ $A$ can be covered by countably many clopen sets $(\sigma_n)_{n\in\omega}$ of small measure, i.e.
    $$\sum_{n\in\omega}\frac{1}{2^{|\sigma_n|}}<\varepsilon.$$
    If $A$ is closed, it is compact (since the Cantor space is compact). Hence, the countable cover from the above fact can be replaced with a finite cover. Therefore, a closed measure zero set can be coded by a tree $T$ satisfying the condition
    \begin{multline*}
        \big(\forall k\in\omega\big)\big(\exists (\sigma_i)_i\in(\seq2)^{<\omega}\big)\big(\sum_n\frac{1}{2^{|\sigma_n|}}<\frac{1}{k+1}\land\\\big(\forall\tau\in T\big)\big(\exists i\big)\big(\sigma_i\subseteq\tau\lor\tau\subseteq\sigma_i\big)\big).
    \end{multline*}
    As all quantifiers used in the above condition are through countable sets, the set of all such $T$s is clearly Borel.
\end{proof}

\end{document}